\documentclass{birkjour}

\usepackage{graphicx}

 \newtheorem{thm}{Theorem}[section]

 \theoremstyle{definition}
 
 \theoremstyle{remark}

 \numberwithin{equation}{section}

\begin{document}

\title[ Bicomplex k-Fibonacci quaternions]{
\\ \\ 
Bicomplex k-Fibonacci quaternions}

\author[F\"{u}gen Torunbalc{\i} Ayd{\i}n]{F\"{u}gen Torunbalc{\i}  Ayd{\i}n}
\address{%
Yildiz Technical University\\
Faculty of Chemical and Metallurgical Engineering\\
Department of Mathematical Engineering\\
Davutpasa Campus, 34220\\
Esenler, Istanbul,  TURKEY}

\email{ftorunay@gmail.com ; faydin@yildiz.edu.tr}

\thanks{*Corresponding Author}

\keywords{ Bicomplex number; k-Fibonacci number; bicomplex k-Fibonacci number; k-Fibonacci quaternion; bicomplex k-Fibonacci quaternion. }

\begin{abstract}
In this paper, bicomplex k-Fibonacci quaternions are defined. Also, some algebraic properties of bicomplex k-Fibonacci quaternions which are connected with bicomplex numbers and k-Fibonacci numbers are investigated. Furthermore, the Honsberger identity, the d'Ocagne's identity, Binet's formula, Cassini's identity, Catalan's identity for these quaternions are given.
\end{abstract}

\maketitle

\section{Introduction}
Many kinds of generalizations of the Fibonacci sequence have been presented in the literature \cite{9}. In 2007, the k-Fibonacci sequence $\{F_{k,n}\}_{n\in\mathbb{N}}$ is defined by Falcon and Plaza \cite{4} as follows
\begin{equation}\label{E1}
\left\{\begin{array}{rl}
{{F}_{k,0}}=&0,\,\,{{F}_{k,1}}=1 \\
{{F}_{k,n+1}}=&k\,{{F}_{k,n}}+\,{{F}_{k,n-1}},\,\ n\geq 1 \\

or \\

\{{F}_{k,n}\}_{n\in\mathbb{N}}=&\{\,0,\,1,\,k,\,k^2+1,\,k^3+2\,k,\,k^4+3\,k^2+1,...\}\\
\end{array}\right.
\end{equation} \\
Here, ${{k}}$  is a positive real number. Recently, Falcon and Plaza worked on k-Fibonacci numbers, sequences and matrices in \cite{5}, \cite{6}, \cite{7}, \cite{8}.\\
In 2010, Bolat and K\"{o}se \cite{2} gave properties of k-Fibonacci numbers.\\
In 2014, Catarino \cite{3} obtained some identities for k-Fibonacci numbers. \\ 
In 2015, Ramirez \cite{16} defined the the k-Fibonacci and the k-Lucas quaternions as follows:
\begin{equation*}
\begin{aligned}
{D}_{k,n}=&\{{F}_{k,n}+i\,{F}_{k,n+1}+j\,\,{F}_{k,n+2}+k\,{F}_{k,n+3}\,| {F}_{k,n},\, n-th\,\, \\ & \text{ k-Fibonacci number} \},
\end{aligned}
\end{equation*} and
\begin{equation*}
\begin{aligned}
{P}_{k,n}=&\{{L}_{k,n}+i\,{L}_{k,n+1}+j\,\,{L}_{k,n+2}+k\,{L}_{k,n+3}\,| {L}_{k,n},\, n-th\,\, \\ & \text{ k-Lucas number} \}
\end{aligned}
\end{equation*}
where ${\,i,\,j,\,k\,}$ satisfy the multiplication rules
\begin{equation*}
{i}^{2}={j}^{2}={k}^{2}=-1\,,\ \ i\ j=-j\ i=k\,,\quad j\ k=-k \ j=i\,,\quad k\ i=-i\ k=j\,.
\end{equation*} 
\par In 2015, Polatl{\i} gave Catalan’s identity for the k-Fibonacci quaternions \cite{13}.
\par In 2016, Polatl{\i}, K{\i}z{\i}ılate\c{s} and Kesim \cite{14} defined split k-Fibonacci and split k-Lucas quaternions $({M}_{k,n})$ and $({N}_{k,n})$ respectively as follows:
\begin{equation*}
\begin{aligned}
{M}_{k,n}=&\{{F}_{k,n}+i\,{F}_{k,n+1}+j\,\,{F}_{k,n+2}+k\,{F}_{k,n+3}\,| {F}_{k,n},\, n-th\,\, \\ & \text{ k-Fibonacci number} \}
\end{aligned}
\end{equation*}
where ${\,i,\,j,\,k\,}$ are split quaternionic units which satisy the multiplication rules
\begin{equation*}
{i}^{2}=-1,\,{j}^{2}={k}^{2}=i\ j\ k=1\,,\ \ i\,j=-j\, i=k,\,j\,k=-k\,j=-i,\,k\,i=-i\,k=j.
\end{equation*}
 \par  In 1892, bicomplex numbers were introduced by Corrado Segre, for the first time \cite{20}. In 1991, G. Baley Price, the bicomplex numbers gave in his book based on multicomplex spaces and functions \cite{15}. In recent years, fractal structures of this numbers are studied \cite{17}, \cite{18}, \cite{19}, \cite{11}. In 2015, Karaku{\c{s}}, S{\i}dd{\i}ka {\"O}zkald{\i} and Aksoyak, Ferdag Kahraman worked on generalized bicomplex numbers and Lie Groups \cite{10}. The set of bicomplex numbers can be expressed by a basis $\{1\,,i\,,j\,,i\,j\,\}$ as, 
\begin{equation}\label{E2}
\begin{aligned}
\mathbb{C}_2=\{\, q=q_1+iq_2+jq_3+ijq_4 \ | \ q_1,q_2,q_3,q_4\in \mathbb R\}
\end{aligned}
\end{equation}
where $i$,$j$ and $ij$ satisfy the conditions 
\begin{equation*}
i^2=-1,\,\,\,j^2=-1,\,\,\,i\,j=j\,i.  
\end{equation*} \,
\par A set of bicomplex numbers $\mathbb{C}_2$ is a real vector space with the addition and scalar multiplication operations. The vector space $\mathbb{C}_2$ equipped with bicomplex product is a real  associative algebra Table 1. Also, the vector space together with properties of multiplication and product of the bicomplex numbers is an commutative algebra. Furthermore, three different conjugations can operate on bicomplex numbers \cite{17},\cite{18} as follows: 
\begin{equation}\label{E3}
\begin{aligned}
q=q_1+i\,q_2+j\,q_3+i\,j\,q_4=(q_1+iq_2)+j\,(q_3+iq_4),\,\, q\in{\mathbb{C}_{2}}\\ 
{q_i}^*=q_1-iq_2+jq_3-ijq_4=(q_1-iq_2)+j\,(q_3-iq_4),\\
{q_j}^*=q_1+iq_2-jq_3-ijq_4=(q_1+iq_2)-j\,(q_3+iq_4),\\
{q_{ij}}^*=q_1-iq_2-jq_3+ijq_4=(q_1-iq_2)-j\,(q_3-iq_4).
\end{aligned}
\end{equation} \\
The norm of the bicomplex numbers is defined as 
\begin{equation}\label{E4}
\begin{aligned}
{{N}_{q}}_{i}=\left\| {q\times{q}_{i}} \right\|=\sqrt[]{\left|{q}_{1}^2+{q}_{2}^2-{q}_{3}^2-{q}_{4}^2+2\,j\,({q}_{1}{q}_{3}+{q}_{2}{q}_{4})\right|}, \\ 
{{N}_{q}}_{j}=\left\| {q\times{q}_{j}} \right\|=\sqrt[]{\left|{q}_{1}^2-{q}_{2}^2+{q}_{3}^2-{q}_{4}^2+2\,i\,({q}_{1}{q}_{2}+{q}_{3}{q}_{4})\right|}, \\ 
{{N}_{q}}_{i\,j}=\left\| {q\times{q}_{i\,j}} \right\|=\sqrt[]{\left|{q}_{1}^2+{q}_{2}^2+{q}_{3}^2+{q}_{4}^2+2\,i\,j\,({q}_{1}{q}_{4}-{q}_{2}{q}_{3})\right|}.  
\end{aligned}
\end{equation} 
\begin{table}[ht]
\caption{Multiplication scheme of bicomplex numbers} 
\centering 
\begin{tabular}{c c c c c} 
\\
\hline 
x & 1 & i & j & i\,j \\ [0.5ex] 
\hline 
1 & 1 & i & j & i\,j \\ 
i & i & -1 & i\,j & -j \\
j & j & i\,j & -1 & -i \\
i\,j & i\,j & -j & -i & 1 \\
 [1ex] 
\hline 
\end{tabular}
\label{table:nonlin} 
\end{table}
\\
In 2015, the bicomplex Fibonacci and Lucas numbers defined by Nurkan and G\"{u}ven \cite{12} as follows
\begin{equation}\label{E5}
{BF}_{n}={F}_{n}+i\,{F}_{n+1}+j\,{F}_{n+2}+k\,{F}_{n+3} 
\end{equation}
and
\begin{equation}\label{E6}
{BL}_{n}={L}_{n}+i\,{L}_{n+1}+j\,{L}_{n+2}+k\,{L}_{n+3}
\end{equation}
where the basis $\{1,\,i,\,j,\,k \}$ satisfy the conditions 
\begin{equation*}
i^2=j^2=-1,\,\,k^2=1,\,i\,j=j\,i=k,\,j\,k=k\,j=-i,\,i\,k=k\,i=-j.  
\end{equation*}
In 2018, the bicomplex Fibonacci quaternions defined by Ayd{\i}n Torunbalc{\i} \cite{1} as follows
\begin{equation}\label{E7}
{Q_F}_{n}={F}_{n}+i\,{F}_{n+1}+j\,{F}_{n+2}+i\,j\,{F}_{n+3} 
\end{equation}
where the basis $\{1,\,i,\,j,\,i\,j \}$ satisfy the conditions 
\begin{equation*}
i^2=-1,\,\,j^2=-1,\,\,i\,j=j\,i,\,\,(i\,j)^2=1.  
\end{equation*}
In this paper, the bicomplex k-Fibonacci quaternions and the bicomplex k-Lucas quaternions will be defined respectively, as follows  
\begin{equation*}
\begin{aligned}
{\mathbb{BC}}^{F_{k,n}}=\{\,{Q_F}_{k,n}=&{F}_{k,n}+i\,{F}_{k,n+1}+j\,{F}_{k,n+2}+i\,j\,{F}_{k,n+3}\,| \, {F}_{k,n},\, nth\, \\ & \text{k-Fibonacci number} \}
\end{aligned}
\end{equation*}
and
\begin{equation*}
\begin{aligned}
{\mathbb{BC}}^{L_{k,n}}=\{\,{{Q}_L}_{k,n}=&{L}_{k,n}+i\,{L}_{k,n+1}+j\,{L}_{k,n+2}+i\,j\,{L}_{k,n+3}\,| \, {L}_{k,n},\, nth\, \\ & \text{k-Lucas number} \}
\end{aligned}
\end{equation*}
where
\begin{equation*}
i^2=-1,\,\,j^2=-1,\,\,i\,j=j\,i,\,\,(i\,j)^2=1.  
\end{equation*}
The aim of this work is to present in a unified manner a variety of algebraic properties of the bicomplex k-Fibonacci quaternions as well as both the bicomplex numbers and k-Fibonacci numbers. In particular, using three types of conjugations, all the properties established for bicomplex numbers and bicomplex k-Fibonacci numbers are also given for the bicomplex k-Fibonacci quaternions. In addition, Binet's Formula, the Honsberger identity, the d'Ocagne's identity, Cassini's identity and Catalan's identity for these quaternions are given.

\section{The bicomplex k-Fibonacci numbers}
The bicomplex k-Fibonacci and k-Lucas numbers can be define by  with the basis $\{1,\,i,\,j,\,i\,j\,\}$, where $i$,\,\,$j\,$ \,and\, $i\,j\,$ satisfy the conditions 
\begin{equation*}
i^2=-1,\,\,j^2=-1,\,\,i\,j=j\,i,\,\,(i\,j)^2=1.  
\end{equation*}
as follows
\begin{equation}\label{F1}
\begin{aligned}
{\mathbb{BC}F_{k,n}}=&({F}_{k,n}+i\,{F}_{k,n+1})+j\,({F}_{k,n+2}+i\,{F}_{k,n+3}) \\
=& {F}_{k,n}+i\,{F}_{k,n+1}+j\,{F}_{k,n+2}+i\,j\,{F}_{k,n+3}
\end{aligned}
\end{equation}
and
\begin{equation}\label{F2}
\begin{aligned}
{\mathbb{BC}L_{k,n}}=&({L}_{k,n}+i\,{L}_{k,n+1})+j \,({L}_{k,n+2}+i\,{L}_{k,n+3}) \\
=& {L}_{k,n}+i\,{L}_{k,n+1}+j\,{L}_{k,n+2}+i\,j\,{L}_{k,n+3}.
\end{aligned}
\end{equation}
 
The addition and subtraction of two bicomplex k-Fibonacci numbers are defined by 
\begin{equation}\label{F3}
\begin{array}{rl}
{\mathbb{BC}F}_{k,n}\pm{\mathbb{BC}F}_{k,m}=&({F}_{k,n}\pm{F}_{k,m})+i\,({F}_{k,n+1}\pm{F}_{k,m+1}) \\
&+j\,({F}_{k,n+2}\pm{F}_{k,m+2})+i\,j\,({F}_{k,n+3}\pm{F}_{k,m+3}) \\
\end{array}
\end{equation}

The multiplication of two bicomplex k-Fibonacci numbers is defined by
\begin{equation}\label{F4}
\begin{array}{rl}
{\mathbb{BC}F_{k,n}}\times\,{\mathbb{BC}F_{k,m}}=&({F}_{k,n}\,{F}_{k,m}-{F}_{k,n+1}\,{F}_{k,m+1} \\
&-{F}_{k,n+2}\,{F}_{k,m+2}-{F}_{k,n+3}\,{F}_{k,m+3}) \\
&+i\,({F}_{k,n}\,{F}_{k,m+1}+{F}_{k,n+1}\,{F}_{k,m} \\
&-{F}_{k,n+2}\,{F}_{k,m+3}-{F}_{k,n+3}\,{F}_{k,m+2}) \\
&+j\,({F}_{k,n}\,{F}_{k,m+2}+{F}_{k,n+2}\,{F}_{k,m} \\
&-{F}_{k,n+1}\,{F}_{k,m+3}-{F}_{k,n+3}\,{F}_{k,m+1}) \\
&+i\,j\,({F}_{k,n}\,{F}_{k,m+3}+{F}_{k,n+3}\,{F}_{k,m} \\
&+{F}_{k,n+1}\,{F}_{k,m+2}+{F}_{k,n+2}\,{F}_{k,m+1}) \\
=&{\mathbb{BC}F_{k,m}}\times\,{\mathbb{BC}F_{k,n}}\,.
\end{array}
\end{equation} 

\section{The bicomplex k-Fibonacci quaternions}

In this section, firstly the bicomplex k-Fibonacci quaternions will be defined. The bicomplex k-Fibonacci quaternions are defined by using the bicomplex numbers and k-Fibonacci numbers as follows 
\begin{equation}\label{G1}
\begin{aligned}
\mathbb{BC}^{F_{k,n}}=\{\,{{Q}_F}_{k,n}=&{F}_{k,n}+i\,{F}_{k,n+1}+j\,{F}_{k,n+2}+i\,j\,{F}_{k,n+3}\,|\, {F}_{k,n},\, n-th \\
& \quad \quad \quad \quad \,\text{k-Fib. number} \, \}
\end{aligned}
\end{equation}
where
\begin{equation*}
i^2=-1,\,\,j^2=-1,\,\,i\,j=j\,i,\,\,(i\,j)^2=1.  
\end{equation*}

Let ${Q_F}_{k,n}$ and ${Q_F}_{k,m}$ be two bicomplex k-Fibonacci quaternions such that
\begin{equation}\label{G2} 
{Q_F}_{k,n}={F}_{k,n}+i\,{F}_{k,n+1}+j\,{F}_{k,n+2}+i\,j\,{F}_{k,n+3}
\end{equation}
and 
\begin{equation}\label{G3}
{Q_F}_{k,m}={F}_{k,m}+i\,{F}_{k,m+1}+j\,{F}_{k,m+2}+i\,j\,{F}_{k,m+3}.
\end{equation}

The addition and subtraction of two bicomplex k-Fibonacci  quaternions are defined in the obvious way,  
\begin{equation}\label{G4}
 \begin{array}{rl}
{Q_F}_{k,n}\pm{Q_F}_{k,m}=&({F}_{k,n}+i\,{F}_{k,n+1}+j\,{F}_{k,n+2}+i\,j\,{F}_{k,n+3}) \\ &\pm ({F}_{k,m}+i\,{F}_{k,m+1}+j\,{F}_{k,m+2}+i\,j\,{F}_{k,m+3}) \\
=&({F}_{k,n}\pm{F}_{k,m})+i\,({F}_{k,n+1}\pm{F}_{k,m+1}) \\
&+j\,({F}_{k,n+2}\pm{F}_{k,m+2})+ i\,j\,({F}_{k,n+3}\pm{F}_{k,m+3}).
\end{array}
\end{equation}

Multiplication of two bicomplex k-Fibonacci  quaternions is defined by 
\begin{equation}\label{G5}
\begin{array}{rl}
{Q_F}_{k,n}\times\,{Q_F}_{k,m}=&({F}_{k,n}+i\,{F}_{k,n+1}+j\,{F}_{k,n+2}+i\,j\,{F}_{k,n+3}) \\&\,({F}_{k,m}+i\,{F}_{k,m+1}+j\,{F}_{k,m+2}+i\,j\,{F}_{k,m+3}) \\
=&[{F}_{k,n}\,{F}_{k,m}-{F}_{k,n+1}\,{F}_{k,m+1} \\
&\quad \quad-{F}_{k,n+2}\,{F}_{k,m+2}+{F}_{k,n+3}\,{F}_{k,m+3}] \\
&+i\,[{F}_{k,n}\,{F}_{k,m+1}+{F}_{k,n+1}\,{F}_{k,m} \\
&\quad \quad-{F}_{k,n+2}\,{F}_{k,m+3}-{F}_{k,n+3}\,{F}_{k,m+2}] \\
&+j\,[{F}_{k,n}\,{F}_{k,m+2}-{F}_{k,n+1}\,{F}_{k,m+3} \\
&\quad \quad+{F}_{k,n+2}\,{F}_{k,m}-{F}_{k,n+3}\,{F}_{k,m+1}] \\
&+i\,j\,[{F}_{k,n}\,{F}_{k,m+3}+{F}_{k,n+1}\,{F}_{k,m+2} \\
&\quad \quad+{F}_{k,n+2}\,{F}_{k,m+1}+{F}_{k,n+3}\,{F}_{k,m}] \\
=&{Q_F}_{k,m}\times\,{Q_F}_{k,n}\,.
\end{array}
\end{equation}

The scaler and the bicomplex vector parts of the bicomplex k-Fibonacci  quaternion $({Q_F}_{k,n})$ are denoted by 
\begin{equation}\label{G6}
{S}_{\,{{Q}_F}_{k,n}}={F}_{k,n} \ \ \text{and} \ \ \ {V}_{\,{{Q}_F}_{k,n}}=i\,{F}_{k,n+1}+j\,{F}_{k,n+2}+i\,j\,{F}_{k,n+3}.	
\end{equation}
Thus, the bicomplex k-Fibonacci quaternion ${Q_F}_{k,n}$  is given by 
\begin{equation*}
{Q_F}_{k,n}={S}_{\,{Q_F}_{k,n}}+{V}_{\,{Q_F}_{k,n}}. 
\end{equation*}           
Three kinds of conjugation can be defined for bicomplex numbers \cite{18}. Therefore, conjugation of the bicomplex k-Fibonacci quaternion is defined in five different ways as follows
\begin{equation} \label{G7}
\begin{aligned}
({Q_F}_{k,n})^{*_1}=&{F}_{k,n}-i\,{F}_{k,n+1}+j\,{F}_{k,n+2}-i\,j\,{F}_{k,n+3}, \\
\end{aligned}
\end{equation}
\begin{equation} \label{G8}
\begin{aligned}
({Q_F}_{k,n})^{*_2}=&{F}_{k,n}+i\,{F}_{k,n+1}-j\,{F}_{k,n+2}-i\,j\,{F}_{k,n+3}, \\ 
\end{aligned}
\end{equation}
\begin{equation} \label{G9}
\begin{aligned}
({Q_F}_{k,n})^{*_3}=&{F}_{k,n}-i\,{F}_{k,n+1}-j\,{F}_{k,n+2}+i\,j\,{F}_{k,n+3}, \\
\end{aligned}
\end{equation}
\\
In the following theorem, some properties related to the conjugations of the bicomplex k-Fibonacci quaternions are given.
\begin{thm}
Let $({Q_F}_{k,n})^{*_1}$,\,$({Q_F}_{k,n})^{*_2}$ and \,$({Q_F}_{k,n})^{*_3}$,\, be three kinds of conjugation of the bicomplex k-Fibonacci quaternion. In this case, we can give the following relations:
\begin{equation}\label{G10}
{Q_F}_{k,n}\,({Q_F}_{k,n})^{*_1}=F_{k,2n+1}-F_{k,2n+5}+2\,j\,{F}_{k,2n+3}, \\
\end{equation}
\begin{equation}\label{G11}
\begin{array}{rl}
{Q_F}_{k,n}\,({Q_F}_{k,n})^{*_2}=&(F_{k,n}^2-F_{k,n+1}^2+F_{k,n+2}^2-F_{k,n+3}^2) \\
&+2\,i\,(2\,{F}_{k,n}\,{F}_{k,n+1}+k\,{F}_{k,2n+3}), \\
\end{array}
\end{equation}
\begin{equation}\label{G12}
{Q_F}_{k,n}\,({Q_F}_{k,n})^{*_3}=(F_{k,2n+1}+{F}_{k,2n+5})+2\,i\,j\,(-1)^{n+1}\,k, \\
\end{equation}
\end{thm}
\begin{proof}

(\ref{G10}): By the Eq.(3.2) and (3.7) we get,
\begin{equation*}
\begin{array}{rl}
{Q_F}_{k,n}\,({Q_F}_{k,n})^{*_1}=&({F}_{k,n}^2+{{F}_{k,n+1}^2}-{F}_{k,n+2}^2-{F}_{k,n+3}^2) \\
&+2\,j\,(\,{F}_{k,n}\,{F}_{k,n+2}+{F}_{k,n+1}\,{F}_{k,n+3}) \\
=&{F}_{k,2n+1}-{F}_{k,2n+5}+2\,j\,{F}_{k,2n+3}. 
\end{array}
\end{equation*}
where ${F}_{k,n}^2+{{F}_{k,n+1}^2}={F}_{k,2n+1}$ and ${F}_{k,n}\,{F}_{k,m-1}+{F}_{k,n+1}\,{F}_{k,m}={F}_{k,n+m}$ are used \cite{4}.\\
\\
(\ref{G11}): By the Eq.(3.2) and (3.8) we get,
\begin{equation*}
\begin{array}{rl}
{Q_F}_{k,n}\,({Q_F}_{k,n})^{*_2}=&({F}_{k,n}^2-{F}_{k,n+1}^2+{F}_{k,n+2}^2-{F}_{k,n+3}^2) \\
&+2\,i\,({F}_{k,n}\,{F}_{k,n+1}+{F}_{k,n+2}\,{F}_{k,n+3}) \\
=&({F}_{k,n}^2-{F}_{k,n+1}^2+{F}_{k,n+2}^2-{F}_{k,n+3}^2) \\
&+2\,i\,(2\,{F}_{k,n}\,{F}_{k,n+1}+k\,{F}_{k,2n+3}). 
\end{array}
\end{equation*}
(\ref{G12}): By the Eq.(3.2) and (3.9) we get,
\begin{equation*}
\begin{array}{rl}
{Q_F}_{k,n}\,({Q_F}_{k,n})^{*_3}=&({F}_{k,n}^2+{F}_{k,n+1}^2+{F}_{k,n+2}^2+{F}_{k,n+3}^2) \\
&+2\,i\,j\,({F}_{k,n}\,{F}_{k,n+3}-{F}_{k,n+1}\,{F}_{k,n+2}) \\
=&{F}_{k,2n+1}+{F}_{k,2n+5}+2\,i\,j\,(-1)^{n+1}\,k. 
\end{array}
\end{equation*}
where ${F}_{k,n}^2-{F}_{k,n-1}\,{F}_{k,n+1}=(-1)^{n+1}$  and ${F}_{k,n}\,{F}_{k,n+3}-{F}_{k,n+1}\,{F}_{k,n+2}=k\,(-1)^{n+1}$ are used \cite{4}. \\
\end{proof} 

Therefore, the norm of the bicomplex k-Fibonacci quaternion ${\,{Q_F}_{k,n}}$ is defined in three different ways as follows
\begin{equation}\label{G13}
\begin{array}{rl}
{N}_({Q_F}_{k,n})^{*_1}=&\|{{Q}_F}_{k,n}\times\,({{Q}_F}_{k,n})^{*_1}\|^2 \\
=&|({F}_{k,n}^2+{F}_{k,n+1}^2)-({F}_{k,n+2}^2+{F}_{k,n+3}^2) \\
&+2\,j\,(\,{F}_{k,n}\,{F}_{k,n+2}+{F}_{k,n+1}\,{F}_{k,n+3})\,| \\
=&|{F}_{k,2n+1}-{F}_{k,2n+5}+2\,j\,{F}_{k,2n+3}|, 
\end{array} 
\end{equation}
\begin{equation}\label{G14}
{\begin{array}{rl}
{N}_({Q_F}_{k,n})^{*_2}=&\|{{Q}_F}_{k,n}\times\,({{Q}_F}_{k,n})^{*_2}\|^2 \\
=&|({F}_{k,n}^2-{F}_{k,n+1}^2)+({F}_{k,n+2}^2-{F}_{k,n+3}^2) \\
&+2\,i\,{F}_{k,n}\,{F}_{k,n+1}+k\,{F}_{k,2n+3}\,|,
\end{array}} 
\end{equation}
\begin{equation}\label{G15}
{\begin{array}{rl}
{N}_({Q_F}_{k,n})^{*_3}=&\|{{Q}_F}_{k,n}\times\,({{Q}_F}_{k,n})^{*_3}\|^2 \\
=&|({F}_{k,n}^2+{F}_{k,n+1}^2)+({F}_{k,n+2}^2+{F}_{k,n+3}^2) \\
&+2\,i\,j\,({F}_{k,n}\,{F}_{k,n+3}-{F}_{k,n+1}\,{F}_{k,n+2})\,| \\
=&|{F}_{k,2n+1}+{F}_{k,2n+5}+2\,i\,j\,(-1)^{n+1}\,k\,|.
\end{array}}  
\end{equation}

In the following theorem, some properties related to the bicomplex k-Fibonacci quaternions are given. 
\begin{thm} 
Let ${\,{Q_F}_{k,n}}$  be the bicomplex k-Fibonacci quaternion. In this case, we can give the following relations: 
\begin{equation}\label{G16}
{Q_F}_{k,n}+k\,{Q_F}_{k,n+1}={Q_F}_{k,n+2},
\end{equation} \,
\begin{equation}\label{G17}
\begin{array}{rl}
({Q_F}_{k,n})^2=& ({F}_{k,n}^2-{F}_{k,n+1}^2-{F}_{k,n+2}^2-{F}_{k,n+3}^2) \\
&+2\,i\,({F}_{k,n}\,{F}_{k,n+1}-{F}_{k,n+2}\,{F}_{k,n+3}) \\
&+2\,j\,({F}_{k,n}\,{F}_{k,n+2}-{F}_{k,n+1}\,{F}_{k,n+3}) \\
\quad \quad \quad \quad &+2\,i\,j\,({F}_{k,n}\,{F}_{k,n+3}+{F}_{k,n+1}\,{F}_{k,n+2}), \\
\end{array}
\end{equation} \,
\begin{equation}\label{G18}
\begin{array}{rl}
({Q_F}_{k,n})^2+({Q_F}_{k,n+1})^2=& {Q_F}_{k,2n+1}+(k\,{F}_{k,2n+6}-{F}_{k,2n+3}) \\
&+i\,({F}_{k,2n+2}-2\,{F}_{k,2n+6}) \\
&+j\,({F}_{k,2n+3}-2\,{F}_{k,2n+5})+i\,j\,(3\,F_{k,2n+4}),
\end{array}
\end{equation}
\begin{equation}\label{G19}
\begin{array}{rl}
({Q_F}_{k,n+1})^2-({Q_F}_{k,n-1})^2=& k\,[\,{Q_F}_{k,2n}-{F}_{k,2n+2}+k\,{F}_{k,2n+5} \\
&+i\,({F}_{k,2n+1}-2\,{F}_{k,2n+5}) \\
&+j\,(-{F}_{k,2n+2}-2\,k\,{F}_{k,2n+3}) \\
&+i\,j\,(\,3\,{F}_{k,2n+3})\,],
\end{array}
\end{equation} \,
\begin{equation}\label{G20}
\begin{array}{rl}
{Q_F}_{k,n}-i\,{Q_F}_{k,n+1}+j\,{{Q}_F}_{k,n+2}-i\,j\,{{Q}_F}_{k,n+3}&={F}_{k,n}+{F}_{k,n+2}-{F}_{k,n+4} \\
&-{F}_{k,n+6}+2\,j\,{L}_{k,n+3}.
\end{array}
\end{equation}
\begin{equation}\label{G21}
\begin{array}{rl}
{Q_F}_{k,n}-i\,{Q_F}_{k,n+1}-j\,{{Q}_F}_{k,n+2}-i\,j\,{{Q}_F}_{k,n+3}&={F}_{k,n}+{F}_{k,n+2}+{F}_{k,n+4} \\
&-{F}_{k,n+6}+2\,i\,{F}_{k,n+5} \\
&+2\,j\,{F}_{k,n+4}-2\,i\,j\,{F}_{k,n+3} \\
=&{L}_{k,n+1}-k\,{F}_{k,n+5}+2\,i\,{F}_{k,n+5} \\
&+2\,j\,{F}_{k,n+4}-2\,i\,j\,{F}_{k,n+3}.
\end{array}
\end{equation}
\end{thm}
\begin{proof} 
(\ref{G16}): By the Eq.(3.2) we get,
\begin{equation*}
\begin{array}{rl}
{Q_F}_{k,n}+k\,{Q_F}_{k,n+1}=&({F}_{k,n}+k\,{F}_{k,n+1})+i\,({F}_{k,n+1}+k\,{F}_{k,n+2}) \\
&+j\,({F}_{k,n+2}+k\,F_{k,n+3})+i\,j\,({F}_{k,n+3}+k\,{F}_{k,n+4}) \\
=&{F}_{k,n+2}+i\,{F}_{k,n+3}+j\,{F}_{k,n+4}+i\,j\,{F}_{k,n+5} \\
=&{Q_F}_{k,n+2}. 
\end{array}
\end{equation*}
(\ref{G17}): By the Eq.(3.2) we get,
\begin{equation*}
\begin{array}{rl}
({Q_F}_{k,n})^2=& ({F}_{k,n}^2-{F}_{k,n+1}^2-{F}_{k,n+2}^2+{F}_{k,n+3}^2) \\
&+2\,i\,({F}_{k,n}\,{F}_{k,n+1}-{F}_{k,n+2}\,{F}_{k,n+3}) \\
&+2\,j\,({F}_{k,n}\,{F}_{k,n+2}-{F}_{k,n+1}\,{F}_{k,n+3}) \\
\quad \quad \quad \quad &+2\,i\,j\,({F}_{k,n}\,{F}_{k,n+3}+{F}_{k,n+1}\,{F}_{k,n+2}). \\
\end{array}
\end{equation*}
(\ref{G18}): By the Eq.(3.2) we get,
\begin{equation*}
\begin{array}{rl}
({Q_F}_{k,n})^2+({{Q}_F}_{k,n+1})^2=&({F}_{k,2n+1}-{F}_{k,2n+3}-{F}_{k,2n+5}+{F}_{k,2n+7}) \\
&+2\,i\,({F}_{k,2n+2}-{F}_{k,2n+6}) \\ 
&+2\,j\,({F}_{k,2n+3}-{F}_{k,2n+5}) \\
&+2\,i\,j\,(2\,{F}_{k,2n+4}) \\ 
=& ({F}_{k,2n+1}+i\,{F}_{k,2n+2}+j\,{F}_{k,2n+3}+i\,j\,{F}_{k,2n+4}) \\
&-{F}_{k,2n+3}-{F}_{k,2n+5}+{F}_{k,2n+7} \\
&+i\,({F}_{k,2n+2}-2\,{F}_{k,2n+6})+j\,({F}_{k,2n+3}-2\,{F}_{k,2n+5}) \\
&+i\,j\,(3\,F_{k,2n+4}) \\
=&{Q_F}_{k,2n+1}+(k\,{F}_{k,2n+6}-{F}_{k,2n+3}) \\
&+i\,({F}_{k,2n+2}-2\,{F}_{k,2n+6})+j\,({F}_{k,2n+3}-2\,{F}_{k,2n+5}) \\
&+i\,j\,(3\,F_{k,2n+4}). 
\end{array}
\end{equation*}
(\ref{G19}): By the Eq.(3.2) we get,
\begin{equation*}
\begin{array}{rl}
({Q_F}_{k,n+1})^2-({{Q}_F}_{k,n-1})^2=& [\,({F}_{k,n+1}^2-{F}_{k,n-1}^2)-({F}_{k,n+2}^2-{F}_{k,n}^2) \\
&\quad \quad-({F}_{k,n+3}^2-{F}_{k,n+1}^2)+({F}_{k,n+4}^2-{F}_{k,n+2}^2)\,] \\
&+2\,i\,[\,({F}_{k,n+1}\,{F}_{k,n+2}-{F}_{k,n-1}\,{F}_{k,n}) \\
&\quad \quad-({F}_{k,n+3}\,{F}_{k,n+4}-{F}_{k,n+1}\,{F}_{k,n+2})\,] \\ 
&+2\,j\,[\,({F}_{k,n+1}\,{F}_{k,n+3}-{F}_{k,n-1}\,{F}_{k,n+1}) \\
&\quad \quad-({F}_{k,n+2}\,{F}_{k,n+4}-{F}_{k,n}\,{F}_{k,n+2})\,] \\
&+2\,i\,j\,[\,({F}_{k,n+1}\,{F}_{k,n+4}-{F}_{k,n-1}\,{F}_{k,n+2}) \\
&\quad \quad+({F}_{k,n+2}\,{F}_{k,n+3}-{F}_{k,n}\,{F}_{k,n+1})\,] \\ 
=& k\,(\,{F}_{k,2n}-k\,{F}_{k,2n+2}-k\,{F}_{k,2n+4}+k\,{F}_{k,2n+6}) \\
&+2\,i\,(k\,{F}_{k,2n+1}-k\,{F}_{k,2n+5})+2\,j\,(-k^2\,{F}_{k,2n+3}) \\
&+2\,i\,j\,(2\,k\,{F}_{k,2n+3}) \\
=& k\,[\,{Q_F}_{k,2n}-{F}_{k,2n+2}+k\,{F}_{k,2n+5} \\
&+i\,({F}_{k,2n+1}-2\,{F}_{k,2n+5}) \\
&+j\,(-{F}_{k,2n+2}-2\,k\,{F}_{k,2n+3}) \\
&+i\,j\,(\,3\,{F}_{k,2n+3})\,] . 
\end{array}
\end{equation*}
(\ref{G20}): By the Eq.(3.2) we get,
\begin{equation*}
\begin{array}{rl}
{Q_F}_{k,n}-i\,{Q_F}_{k,n+1}-j\,{Q_F}_{k,n+2}-i\,j\,{Q_F}_{k,n+3}&=({F}_{k,n}+{F}_{k,n+2}-k\,{F}_{k,n+5}) \\
&+2\,i\,{F}_{k,n+5}+2\,j\,{F}_{k,n+4} \\
&-2\,i\,j\,{F}_{k,n+3} \\
=&{L}_{k,n+1}-k\,{F}_{k,n+5}+2\,i\,{F}_{k,n+5} \\
&+2\,j\,{F}_{k,n+4}-2\,i\,j\,2\,j\,{F}_{k,n+3}.\,
\end{array}
\end{equation*}
(\ref{G21}): By the Eq.(3.2) we get,
\begin{equation*}
\begin{array}{rl}
{Q_F}_{k,n}-i\,{Q_F}_{k,n+1}+j\,{Q_F}_{k,n+2}-i\,j\,{Q_F}_{k,n+3}&=({F}_{k,n}+{F}_{k,n+2}-{F}_{k,n+4} \\
&-{F}_{k,n+6})+2\,j\,({F}_{k,n+2}+{F}_{k,n+4}) \\
=&{L}_{k,n+1}-({F}_{k,n+4}+{F}_{k,n+6}) \\
&+2\,j\,{L}_{k,n+3}.\,
\end{array}
\end{equation*}
\end{proof}
\begin{thm} Let ${Q_F}_{k,n}=({F}_{k,n},{F}_{k,n+1},{F}_{k,n+2},{F}_{k,n+3})$ and \\
${Q_L}_{k,n}=({L}_{k,n},{L}_{k,n+1},{L}_{k,n+2},{L}_{k,n+3})$ be the bicomplex k-Fibonacci quaternion and the bicomplex k-Lucas quaternion respectively. The following relations are satisfied
\begin{equation}\label{G22}
\begin{aligned}
{Q_F}_{k,n+1}+{Q_F}_{k,n-1}={L}_{k,n}+i{L}_{k,n+1}+j\,{L}_{k,n+2}+ij\,{L}_{k,n+3}={Q_L}_{k,n}, \\
\end{aligned}
\end{equation}
\begin{equation}\label{G23}
\begin{aligned}
{Q_F}_{k,n+2}-{Q_F}_{k,n-2}={L}_{k,n}+i{L}_{k,n+1} +j\,{L}_{k,n+2}+ij\,{L}_{k,n+3}=k\,{Q_L}_{k,n}.
\end{aligned}
\end{equation}
\end{thm}
\begin{proof}
\begin{equation*}
\begin{aligned}
\begin{array}{rl}
{Q_F}_{k,n+1}+{Q_F}_{k,n-1}=&({F}_{k,n+1}+{F}_{k,n-1})+i\,({F}_{k,n+2}+{F}_{k,n}) \\
&+j\,({F}_{k,n+3}+{F}_{k,n+1})\\&+i\,j\,({F}_{k,n+4}+{F}_{k,n+2}) \\
=&({L}_{k,n}+i\,{L}_{k,n+1}+j\,{L}_{k,n+2}+i\,j\,{L}_{k,n+3}) \\
=&{Q_L}_{k,n}.
\end{array}
\end{aligned}
\end{equation*}
and
\begin{equation*}
\begin{array}{rl}
{Q_F}_{k,n+2}-{Q_F}_{k,n-2}=&({F}_{k,n+2}-{F}_{k,n-2})+i\,({F}_{k,n+3}-{F}_{k,n-1}) \\
&+j\,({F}_{k,n+4}-{F}_{k,n})+i\,j\,({F}_{k,n+5}-{F}_{k,n+1}) \\
=&k\,({L}_{k,n}+i\,{L}_{k,n+1}+j\,{L}_{k,n+2}+i\,j\,{L}_{k,n+3}) \\
=&k\,{Q_L}_{k,n}.
\end{array}
\end{equation*}
\end{proof} 
\begin{thm} 
For $n,m\ge0$ the Honsberger identity for  the bicomplex k-Fibonacci quaternions ${Q_F}_{k,n}$ and ${Q_F}_{k,m}$  is given by
\begin{equation}\label{G24}
\begin{array}{rl}
{Q_F}_{k,n}\,{Q_F}_{k,m}+{Q_F}_{k,n+1}\,{Q_F}_{k,m+1}=& {Q_F}_{k,n+m+1}-{F}_{k,n+m+3}+k\,{F}_{k,n+m+6}\\
&+i\,({F}_{k,n+m+2}-2\,{F}_{k,n+m+6}) \\
&+j\,({F}_{k,n+m+3}-2\,{F}_{k,n+m+5})\\
&+i\,j\,(3\,{F}_{k,n+m+4}).
\end{array}
\end{equation}
\end{thm}
\begin{proof}
(\ref{G24}): By the Eq.(3.2) we get,
\begin{equation*}
\begin{array}{rl}
{Q_F}_{k,n}\,{Q_F}_{k,m}+{Q_F}_{k,n+1}\,{Q_F}_{k,m+1}=&[\,({F}_{k,n}{F}_{k,m}+{F}_{k,n+1}{F}_{k,m+1}) \\
&\quad \quad-({F}_{k,n+1}{F}_{k,m+1}+{F}_{k,n+2}{F}_{k,m+2}) \\
&\quad \quad-({F}_{k,n+2}{F}_{k,m+2}+{F}_{k,n+3}{F}_{k,m+3}) \\
&\quad \quad+({F}_{k,n+3}{F}_{k,m+3}+{F}_{k,n+4}{F}_{k,m+4})\,] \\
&+i\,[\,({F}_{k,n}{F}_{k,m+1}+{F}_{k,n+1}{F}_{k,m+2}) \\
&\quad \quad+({F}_{k,n+1}{F}_{k,m}+{F}_{k,n+2}{F}_{k,m+1}) \\
&\quad \quad-({F}_{k,n+2}{F}_{k,m+3}+{F}_{k,n+3}{F}_{k,m+4}) \\
&\quad \quad-({F}_{k,n+3}{F}_{k,m+2}+{F}_{k,n+4}{F}_{k,m+3})\,] \\
&+j\,[\,({F}_{k,n}{F}_{k,m+2}+{F}_{k,n+1}{F}_{k,m+3}) \\
&\quad \quad+({F}_{k,n+2}{F}_{k,m}+{F}_{k,n+3}{F}_{k,m+1}) \\
&\quad \quad-({F}_{k,n+1}{F}_{k,m+3}+{F}_{k,n+2}{F}_{k,m+4}) \\
&\quad \quad-({F}_{k,n+3}{F}_{k,m+1}+{F}_{k,n+4}{F}_{k,m+2})\,] \\
&+i\,j\,[\,({F}_{k,n}{F}_{k,m+3}+{F}_{k,n+1}{F}_{k,m+4}) \\
&\quad \quad+({F}_{k,n+1}{F}_{k,m+2}+{F}_{k,n+2}{F}_{k,m+3}) \\
&\quad \quad+({F}_{k,n+2}{F}_{k,m+1}+{F}_{k,n+3}{F}_{k,m+2}) \\
&\quad \quad+({F}_{k,n+3}{F}_{k,m}+{F}_{k,n+4}{F}_{k,m+1})\,] \\
=&({F}_{k,n+m+1}-{F}_{k,n+m+3}-{F}_{k,n+m+5} \\
&+{F}_{k,n+m+7})+2\,i\,({F}_{k,n+m+2}-{F}_{k,n+m+6}) \\
&+2\,j\,({F}_{k,n+m+3}-{F}_{k,n+m+5}) \\
&+i\,j\,(4\,{F}_{k,n+m+4}) \\
=&({F}_{k,n+m+1}+\,i\,{F}_{k,n+m+2}+j\,{F}_{k,n+m+3} \\
&+\,i\,j\,{F}_{k,n+m+4})-{F}_{k,n+m+3}+k\,{F}_{k,n+m+6} \\
&+i\,({F}_{k,n+m+2}-2\,{F}_{k,n+m+6}) \\
&+j\,({F}_{k,n+m+3}-2\,{F}_{k,n+m+5}) \\
&+i\,j\,(3\,{F}_{k,n+m+4}) \\
=&{Q_F}_{k,n+m+1}-{F}_{k,n+m+3}+k\,{F}_{k,n+m+6} \\
&+i\,({F}_{k,n+m+2}-2\,{F}_{k,n+m+6}) \\
&+j\,({F}_{k,n+m+3}-2\,{F}_{k,n+m+3})\\
&+i\,j\,(3\,{F}_{k,n+m+4}).\,
\end{array}
\end{equation*}
where the identity ${F}_{k,n}{F}_{k,m}+{F}_{k,n+1}{F}_{k,m+1}={F}_{k,n+m+1}$ was used \cite{4}.  
\end{proof}
\begin{thm} 
For ${n,m\ge0}$ the D'Ocagne's identity for the bicomplex k-Fibonacci quaternions ${Q_F}_{k,n}$ and ${Q_F}_{k,m}$ is given by
\begin{equation}\label{G25}
\begin{array}{rl}
{Q_F}_{k,n}\,{Q_F}_{k,m+1}-{Q_F}_{k,n+1}\,{Q_F}_{k,m}&=(-1)^m\,{F}_{k,n-m}\,[\,2\,(k^2+2)\,j+(k^3+2\,k)\,i\,j\,]. 
\end{array}
\end{equation}
\end{thm}
\begin{proof} 
(\ref{G25}): By the Eq.(3.2) we get,
\begin{equation*}
\begin{array}{rl}
{Q_F}_{k,n}\,{Q_F}_{k,m+1}-{Q_F}_{k,n+1}\,{Q_F}_{k,m}=&[\,({F}_{k,n}{F}_{k,m+1}-{F}_{k,n+1}{F}_{k,m}) \\
&\quad-({F}_{k,n+1}{F}_{k,m+2}-{F}_{k,n+2}{F}_{k,m+1}) \\
&\quad-({F}_{k,n+2}{F}_{k,m+3}-{F}_{k,n+3}{F}_{k,m+2}) \\
&\quad+({F}_{k,n+3}{F}_{k,m+4}-{F}_{k,n+4}{F}_{k,m+3})\,] \\
&+\,i\,[\,({F}_{k,n}{F}_{k,m+2}-{F}_{k,n+1}{F}_{k,m+1})\\
&\quad+({F}_{k,n+1}{F}_{k,m+1}-{F}_{k,n+2}{F}_{k,m}) \\
&\quad-({F}_{k,n+3}{F}_{k,m+3}-{F}_{k,n+4}{F}_{k,m+2})\,] \\
&+\,j\,[\,({F}_{k,n}{F}_{k,m+3}-{F}_{k,n+1}{F}_{k,m+2})\\
&\quad+({F}_{k,n+2}{F}_{k,m+1}-{F}_{k,n+3}{F}_{k,m}) \\
&\quad-({F}_{k,n+1}{F}_{k,m+4}-{F}_{k,n+2}{F}_{k,m+3})\\
&\quad-({F}_{k,n+3}{F}_{k,m+2}-{F}_{k,n+4}{F}_{k,m+1})\,] \\
&+\,i\,j\,[\,({F}_{k,n}{F}_{k,m+4}-{F}_{k,n+1}{F}_{k,m+3})\\
&\quad+({F}_{k,n+1}{F}_{k,m+3}-{F}_{k,n+2}{F}_{k,m+2}) \\
&\quad+({F}_{k,n+2}{F}_{k,m+2}-{F}_{k,n+3}{F}_{k,m+1})\\
&\quad+({F}_{k,n+3}{F}_{k,m+1}{F}_{k,n+4}{F}_{k,m})\,] \\
=&(-1)^m\,{F}_{k,n-m}\,[\,2\,(k^2+2)\,j \\
&\quad+(k^3+2\,k)\,i\,j\,].\,
\end{array}
\end{equation*}
where the identity ${F}_{k,m}{F}_{k,n+1}-{F}_{k,m+1}{F}_{k,n}=(-1)^{n}{F}_{k,m-n}$  is used \cite{5}.
\end{proof}
\begin{thm} 
Let ${Q_F}_{k,n}$ be the bicomplex k-Fibonacci quaternion.Then, we have the following identities
\begin{equation}\label{G26}
\sum\limits_{s=1}^{n}{\,{Q_F}_{k,s}}=\frac{1}{k}\,(\,{Q_F}_{k,n+1}+{Q_F}_{k,n}-{Q_F}_{k,1}-{Q_F}_{k,0}\,),
\end{equation}
\begin{equation}\label{G27}
\sum\limits_{s=1}^{n}{\,{Q_F}_{k,2s-1}}=\frac{1}{k}({Q_F}_{k,2n}-{Q_F}_{k,0}),
\end{equation}
\begin{equation}\label{G28}
\sum\limits_{s=1}^{n}{\,{Q_F}_{2s}}=\frac{1}{k}({Q_F}_{k,2n+1}-{Q_F}_{k,1}).
\end{equation} 
\end{thm}
\begin{proof}
(\ref{G26}) Since $\sum\nolimits_{i=1}^{n}{{F}_{k,i}}=\frac{1}{k}({F}_{k,n+1}+{F}_{k,n}-1)$  \cite{4}, \, we get
\begin{equation*}
\begin{aligned}
  & \sum\limits_{s=1}^{n}\,{{Q_F}_{k,s}}=\sum\limits_{s=1}^{n}{{F}_{k,s}}+i\,\sum\limits_{s=1}^{n}{{F}_{k,s+1}}+\varepsilon\,\sum\limits_{s=1}^{n}{{F}_{k,s+2}}+i\,j\,\sum\limits_{s=1}^{n}{{F}_{k,s+3}} \\
& \quad  =\frac{1}{k}\{({F}_{k,n+1}+{F}_{k,n}-1)+i\,({F}_{k,n+2}+{F}_{k,n+1}-k-1) \\
& \quad \quad +j\,[\,{F}_{k,n+3}+{F}_{k,n+2}-(k^2+1)-k\,] \\
& \quad \quad +i\,j\,[\,{F}_{k,n+4}+{F}_{k,n+3}-(k^3+2\,k)-(k^2+1)\,]\} \\
 & \quad  =\frac{1}{k}\{({F}_{k,n+1}+i\,{F}_{k,n+2}+j\,{F}_{k,n+3}+i\,j\,{F}_{k,n+4}) \\
& \quad \quad +({F}_{k,n}+i\,{F}_{k,n+1}+j\,{F}_{k,n+2}+i\,j\,{F}_{k,n+3}) \\
& \quad \quad -[\,1+(k+1)\,i+(k^2+k+1)\,j+(k^3+k^2+2\,k+1)\,]\}\\
 & \quad =\frac{1}{k}\,(\,{Q_F}_{k,n+1}+{Q_F}_{k,n}-{Q_F}_{k,1}-{Q_F}_{k,0}\,).
\end{aligned}
\end{equation*}
(\ref{G27}): Using $\sum\limits_{i=1}^{n}{F}_{k,2i+1}=\frac{1}{k}{F}_{k,2n+2}$ \, and \, $\sum\limits_{i=1}^{n}{F}_{k,2i}=\frac{1}{k}({F}_{k,2n+1}-1)$ \,  \cite{4}, \, we get
\begin{equation*}
\begin{array}{rl}
  \sum\limits_{s=1}^{n}{\,{{Q}_F}_{k,2s-1}}= &\frac{1}{k}\,\{\,({F}_{k,2n})+i\,({F}_{k,2n+1}-1)+j\,({F}_{k,2n+2}-k) 
	\\ &+i\,j\,({F}_{k,2n+3}-(k^2+1)\,)\}  
	\\=&\frac{1}{k}\,\{[\,{F}_{k,2n}+i\,{F}_{k,2n+1}+j\,{F}_{k,2n+2}+i\,j\,{F}_{k,2n+3}]
  \\ &-\,(\,0+i+k\,j+2\,(k^2+1)\,i\,j\,)\}
	\\=&\frac{1}{k}\,\{{Q_F}_{k,2n}-[{F}_{k,0}+i\,{F}_{k,1}+\,j\,{F}_{k,2}+\,i\,j\,{F}_{k,3}]\}
  \\=&\frac{1}{k}\,(\,{Q_F}_{k,2n}-\,{Q_F}_{k,0}\,)\,.
\end{array}
\end{equation*}
(\ref{G28}): Using $\sum\limits_{i=1}^{n}{F}_{k,2i}=\frac{1}{k}({F}_{2n+1}-1)$ \, \cite{4}, \, we obtain
\begin{equation*}
\begin{array}{rl}
\sum\limits_{s=1}^{n}{\,{{Q}_F}_{k,2s}}= & \frac{1}{k}\,\{({F}_{k,2n+1}-1)+i\,({F}_{k,2n+2}-k)+j\,({F}_{k,2n+3}-(k^2+1))
\\& \quad \quad +i\,j\,({F}_{k,2n+4}-(k^2+2\,k))\}  
\\=&\frac{1}{k}\,\{({F}_{k,2n+1}+i\,{F}_{k,2n+2}+j\,{F}_{k,2n+3}+i\,j\,{F}_{k,2n+4})
\\& \quad \quad -\,(1+k\,i\,+(k^2+1)\,j\,+(k^3+2\,k)\,i\,j\,)\,\}
\\=&\frac{1}{k}\,\{\,{Q_F}_{k,2n+1}-({F}_{k,1}+i\,{F}_{k,2}+j\,{F}_{k,3}+i\,j\,{F}_{k,4})\}
  \\=&\frac{1}{k}\,(\,{Q_F}_{k,2n+1}-{Q_F}_{k,1}\,)\,. 
\end{array}
\end{equation*}
\end{proof}
\begin{thm} \textbf{Binet's Formula}
\\ 
Let ${\,{Q_F}_{k,n}}$ be the bicomplex k-Fibonacci quaternion. For $n\ge 1$, Binet's formula for these quaternions is as follows:
\begin{equation}\label{G29}
{Q_F}_{k,n}=\frac{1}{\alpha -\beta }\left( \,\hat{\alpha }\,{\alpha}^{n}-\hat{\beta \,}\,{\beta }^{n} \right)\,
\end{equation}
where   
\begin{equation*}
\begin{array}{l}
\hat{\alpha }=1+i\,{\alpha}+j\,{\alpha}^2+i\,j\,{\alpha}^3,\,\,\,\,\, \alpha=\frac{k+\sqrt{k^2+4}}{2},
\end{array}
\end{equation*}

\begin{equation*}
\begin{array}{l}
\hat{\beta }=1+i\,{\beta}+j\,{\beta}^2+i\,j\,{\beta}^3,\,\,\,\,\, \beta=\frac{k-\sqrt{k^2+4}}{2},
\end{array}
\end{equation*}
$\alpha +\beta =k\ ,\ \ \alpha -\beta =\sqrt{k^2+4}\,,\ \ \alpha \beta =-1$.\\
\end{thm}
\begin{proof}

By using the Binet formula for k-Fibonacci number \cite{5}, we obtain  
\begin{equation*}
\begin{array}{rl}
{Q_F}_{k,n}=&{F}_{k,n}+i\,{F}_{k,n+1}+j\,{F}_{k,n+2}+i\,j\,{F}_{k,n+3} \\
\\
=&\frac{\alpha^n -\beta^n }{\sqrt{k^2+4}}+i\,(\frac{\alpha^{n+1} -\beta^{n+1}\,}{\sqrt{k^2+4}})+j\,(\frac{\alpha^{n+2} -\beta^{n+2}\,}{\sqrt{k^2+4}})+i\,j\,(\frac{\alpha^{n+3} -\beta^{n+3}\,}{\sqrt{k^2+4}}) \\
\\
=&\frac{\alpha^{n}\,(1+i\,\alpha+j\,\alpha^2+i\,j\,\alpha^3)-\beta^{n}\,(1+i\,\beta+j\,\beta^2+i\,j\,\beta^3)\,}{\sqrt{k^2+4}} \\
\\
=&\frac{1}{\sqrt{k^2+4}}( \,\hat{\alpha }\,\alpha^{n}-\hat{\beta}\,\beta^{n}).
\end{array}
\end{equation*}
where $\hat{\alpha }=\,1+i\,\alpha +j\,\alpha^2+i\,j\,\alpha^3, \, \, \hat{\beta }=\,\,1+i\,\beta +j\,\beta^2+i\,j\,\beta^3$. 
\end{proof}
\begin{thm} \textbf{Cassini's Identity}
\\ 
Let ${\,{Q_F}_{k,n}}$ be the bicomplex k-Fibonacci quaternion. For $n\ge 1$, Cassini's identity for ${\,{Q_F}_{k,n}}$ is as follows: 
\begin{equation}\label{G30}
{Q_F}_{k,n-1}\,{Q_F}_{k,n+1}-({Q_F}_{k,n})^2=(-1)^{n}\,[\,2(k^2+2)\,j+\,(k^3+2\,k)\,i\,j\,]. 
\end{equation}
\end{thm}
\begin{proof} 
(\ref{G30}): By using (3.2) we get
\begin{equation*}
\begin{array}{rl}
{Q_F}_{k,n-1}\,{Q_F}_{k,n+1}-(\,{{Q}_F}_{k,n})^2=&\,[\,({F}_{k,n-1}{F}_{k,n+1}-{F}_{k,n}^2) \\
&\quad-({F}_{k,n}{F}_{k,n+2}-{F}_{k,n+1}^2) \\
&\quad-({F}_{k,n+1}{F}_{k,n+3}-{F}_{k,n+2}^2) \\
&\quad+({F}_{k,n+2}{F}_{k,n+4}-{F}_{k,n+3}^2)\,] \\
&+i\,[\,({F}_{k,n-1}{F}_{k,n+2}-{F}_{k,n}{F}_{k,n+1}) \\
&\quad-({F}_{k,n+1}{F}_{k,n+4}-{F}_{k,n+2}{F}_{k,n+3})\,] \\
&+j\,[\,({F}_{k,n-1}{F}_{k,n+3}-{F}_{k,n}{F}_{k,n+2}) \\
&\quad-({F}_{k,n}{F}_{k,n+4}-{F}_{k,n+1}{F}_{k,n+3})\\
&\quad+({F}_{k,n+1}{F}_{k,n+1}-{F}_{k,n+2}{F}_{k,n}) \\
&\quad-({F}_{k,n+2}{F}_{k,n+2}-{F}_{k,n+3}{F}_{k,n+1})\,] \\
&+i\,j\,({F}_{k,n-1}{F}_{k,n+4}-{F}_{k,n}{F}_{k,n+3}) \\
=&(-1)^{n}\,{F}_{k,n-m}[\,2(k^2+2)\,j+\,(k^3+2\,k)\,i\,j\,]. 
\end{array}
\end{equation*} 
where the identities of the k-Fibonacci numbers ${F}_{k,n-1}{F}_{k,n+1}-{F}_{k,n}^2=(-1)^{n}$ \cite{5}. Furthermore; 
\begin{equation*}
\left\{\begin{array}{l}
{F}_{k,n-1}{F}_{k,n+2}-{F}_{k,n}{F}_{k,n+1}=k\,(-1)^n\,\\
{F}_{k,n-1}{F}_{k,n+3}-{F}_{k,n}{F}_{k,n+2}=(k^2+1)\,(-1)^n\,,\\
{F}_{k,n+1}{F}_{k,n+3}-{F}_{k,n}{F}_{k,n+4}=(k^2+1)\,(-1)^n\,, \\ 
{F}_{k,n-1}{F}_{k,n+4}-{F}_{k,n}{F}_{k,n+3}=(k^3+2\,k)\,(-1)^n\,
\end{array}\right.
\end{equation*}
are used.   
\end{proof}
\begin{thm} \textbf{Catalan's Identity}
\\ 
Let ${Q_F}_{k,n+r}$ be the bicomplex k-Fibonacci quaternion. For $n\ge 1$, Catalan's identity for ${Q_F}_{k,n+r}$ is as follows: 
\begin{equation}\label{G31}
{Q_F}_{k,n+r-1}\,{Q_F}_{k,n+r+1}-({Q_F}_{k,n+r})^2=(-1)^{n+r}\,[\,2(k^2+2)\,j+\,(k^3+2\,k)\,i\,j\,].
\end{equation}
\end{thm}
\begin{proof}
(\ref{G31}): By using (3.2) we get 
\begin{equation*}
\begin{array}{rl}
{Q_F}_{k,n+r-1}\,{Q_F}_{k,n+r+1}-({Q_F}_{k,n+r})^2&=({F}_{k,n+r-1}{F}_{k,n+r+1}-{F}_{k,n+r}^2) \\
&\quad-({F}_{k,n+r}{F}_{k,n+r+2}-{F}_{k,n+r+1}^2)\\
&\quad({F}_{k,n+r+1}{F}_{k,n+r+3}-{F}_{k,n+r+2}^2) \\
&\quad+({F}_{k,n+r+2}{F}_{k,n+r+4}-{F}_{k,n+r+3}^2) \\ 
&+i\,[\,({F}_{k,n+r-1}{F}_{k,n+r+2})-({F}_{k,n+r}{F}_{k,n+r+1}) \\
&\quad-({F}_{k,n+r+1}{F}_{k,n+r+4}-{F}_{k,n+r+2}{F}_{k,n+r+3})\\ 
&+j\,[\,({F}_{k,n+r-1}{F}_{k,n+r+3}-{F}_{k,n+r}{F}_{k,n+r+2})\\
&\quad-({F}_{k,n+r}{F}_{k,n+r+4}-{F}_{k,n+r+1}{F}_{k,n+r+3}) \\
 &\quad+({F}_{k,n+r+1}{F}_{k,n+r+1}-{F}_{k,n+r+2}{F}_{k,n+r}) \\
 &\quad-({F}_{k,n+r+2}{F}_{k,n+r+2}-{F}_{k,n+r+3}{F}_{k,n+r+1})\,]\\
 &+i\,j\,[\,({F}_{k,n+r-1}{F}_{k,n+r+4}-{F}_{k,n+r}{F}_{k,n+r+3})\\
&\quad+({F}_{k,n+r}{F}_{k,n+r+3}-{F}_{k,n+r+1}{F}_{k,n+r+2}) \\
&\quad+({F}_{k,n+r+2}{F}_{k,n+r+1}-{F}_{k,n+r+3}{F}_{k,n+r})\,] \\
&=(-1)^{n+r}\,[\,2\,(k^2+2)\,j+\,(k^3+2\,k)\,i\,j\,] \\
\end{array}
\end{equation*}
where the identity of the k-Fibonacci numbers ${F}_{k,n+r-1}{F}_{k,n+r+1}-{F}_{k,n+r}^2=(-1)^{n+r}$ is used \cite{5}. Furthermore; 
\begin{equation*}
\left\{\begin{array}{l}
{F}_{k,n+r-1}\,{F}_{k,n+r+2}+{F}_{k,n+r}\,{F}_{k,n+r+1}=(-1)^{n+r}\,k,\\
{F}_{k,n+r-1}\,{F}_{k,n+r+3}-{F}_{k,n+r}\,{F}_{k,n+r+2}=(-1)^{n+r}\,(k^2+1),\\
{F}_{k,n+r+1}\,{F}_{k,n+r+3}-{F}_{k,n+r}\,{F}_{k,n+r+4}=(-1)^{n+r}\,(k^2+1),\\
{F}_{k,n+r-1}\,{F}_{k,n+r+4}-{F}_{k,n+r}\,{F}_{k,n+r+3}=(-1)^{n+r}\,(k^3+2\,k),\\
{F}_{k,n+r}\,{F}_{k,n+r+3}-{F}_{k,n+r+1}\,{F}_{k,n+r+2}=(-1)^{n+r+1}\,k,\\
{F}_{k,n+r+2}\,{F}_{k,n+r+1}-{F}_{k,n+r+3}\,{F}_{k,n+r}=(-1)^{n+r}\,k\,. 
\end{array}\right.
\end{equation*}
are used.
\end{proof} 
\section{Conclusion} 
In this paper, a number of new results on bicomplex k-Fibonacci quaternions were derived.  
This study fills the gap in the literature by providing the bicomplex k-Fibonacci quaternion using definitions of the bicomplex number \cite{18}, k-Fibonacci quaternion \cite{16} and the bicomplex Fibonacci quaternion \cite{1}.

\end{document}